\documentclass[11pt]{article}

\usepackage{amssymb,amsthm,amsmath}
\usepackage[small]{caption}
\usepackage{graphicx}
\usepackage{epsfig}
\usepackage{epsf}
\usepackage{latexsym}
\usepackage{bbm}
\usepackage{float}

\newtheorem{theorem}{Theorem}
\newtheorem{lemma}{Lemma}
\newtheorem{proposition}{Proposition}
\newtheorem{corollary}{Corollary}
\newtheorem{observation}{Observation}

\newcommand{\eps}{\varepsilon}

\newcommand{\NN}{\mathbb{N}}

\def\L{\mathcal{L}}
\def\Q{\mathcal{Q}}

\def\etal{{et~al.}}
\def\ie{{i.e.}}
\def\eg{{e.g.}}

\newcommand{\later}[1]{}
\newcommand{\old}[1]{}

\title{Disjoint empty disks supported by a point set}

\author{Adrian Dumitrescu\thanks{%
Department of Computer Science,
University of Wisconsin--Milwaukee, USA\@.
Email:~\texttt{dumitres@uwm.edu}.
Supported in part by NSF grant DMS-1001667.}
\and
Minghui Jiang\thanks{%
Department of Computer Science,
Utah State University, Logan, USA\@.
Email:~\texttt{mjiang@cc.usu.edu}.}}

\begin{document}

\maketitle

\begin{abstract}
For a planar point-set $P$, let $D(P)$ be the minimum number of
pairwise-disjoint empty disks such that each point in $P$ lies on the
boundary of some disk. 
Further define $D(n)$ as the maximum of $D(P)$ over all
$n$-element point sets. Hosono and Urabe recently conjectured that 
$D(n)=\lceil n/2 \rceil$.
Here we show that $D(n) \geq n/2 + n/236 - O(\sqrt{n})$
and thereby disprove this conjecture.

\medskip
\noindent
\textbf{\small Keywords}:
Empty disk, point-circle incidence, graph, connected component. 

\noindent
\textbf{\small Mathematics Subject Classification (MSC)}:
52C10 Erd\H{o}s problems and related topics of discrete geometry.

\end{abstract}

\section{Introduction} \label{sec:intro}

Hosono and Urabe~\cite{HU11}  have recently considered two families of
disks ``incident'' to a point set, as follows. 
For a given planar point set $P$, let $B(P)$ be the minimum number of
empty disks such that each point in $P$ lies on the boundary of some
disk. Further define the \emph{bubble number} $B(n)$ as the maximum of
$B(P)$ over all $n$-element point sets. 

Analogously, let $D(P)$ be the minimum number of
pairwise-disjoint empty disks such that each point in $P$ lies on the
boundary of some disk. 
Further define the \emph{disjoint bubble number} $D(n)$ as the maximum
of $D(P)$ over all $n$-element point sets. 

A point is said to \emph{support} an empty disk if it lies on the
boundary of the disk. Conversely, we say that the disk is 
\emph{supported} by the point; note that a disk can be supported by
multiple points. A point set is said to \emph{support} a set of empty
disks if each point supports at least one (empty) disk. 
Conversely, we say that a set of empty disks is 
\emph{supported} by the point-set. Such a set of empty disks
is also referred to as a \emph{bubble set} for the given set of points. 

Consider a set of points $P$. Since any set of pairwise-disjoint empty
disks supporting $P$ is simply a set of empty disks supporting $P$, 
we have $B(P) \leq D(P)$ for every $P$, and consequently
\begin{equation} \label{E1}
B(n) \leq D(n), \ \ {\rm for \ every \ }n.
\end{equation} 

Hosono and Urabe~\cite{HU11} have shown the following 
upper and lower bounds for $B(n)$: for every $n \geq 14$,
$ \lceil n/2 \rceil \leq B(n) \leq \lfloor (2n-2)/3) \rfloor$. 
They also conjectured that 
$B(n)=D(n)=\lceil n/2 \rceil$.
Here we show that out of three implied equalities,
$B(n)=D(n)$, $D(n)=\lceil n/2 \rceil$, 
and $B(n)=\lceil n/2 \rceil$, only the third holds,
while the first two are false.

Dillencourt~\cite{Di90} proved that every Delaunay triangulation of
$n$ points ($n$ even) has a perfect matching.
By definition, for every edge in a Delaunay triangulation,
there is an empty disk with the two endpoints of this edge on its
boundary. Consequently, we have the following.

\begin{proposition}\label{prop}
For every $n \ge 1$, $B(n) =\lceil n/2 \rceil$.
\end{proposition}

In contrast,
we show that $D(n) > \lceil n/2 \rceil$ for even $n$ as small as $174$.
\begin{theorem}\label{thm:174}
For every even $n \ge 174$,
$D(n) \ge n/2 + 1$.
\end{theorem}

Proposition~\ref{prop} and Theorem~\ref{thm:174} together imply
that for every even $n \ge 174$,
$B(n) \neq D(n)$ and $D(n) > \lceil n/2 \rceil$,
which disprove the conjecture of Hosono and Urabe~\cite{HU11} that
$B(n)=D(n)=\lceil n/2 \rceil$.
Our main result is the following sharper bound that holds for every $n$.
\begin{theorem}\label{thm:236}
$D(n) \geq n/2 +n/236 - O(\sqrt{n})$. 
\end{theorem}

In a related work, \'Abrego~\etal~\cite{AA+05,AA+09} consider the
problem of matching a given set of $n$ points ($n$ even) by using $n/2$ disks
with exactly two points in each disk.
When the disks could overlap, such a perfect matching always exists if
the $n$ points are in general position.
Similar to Proposition~\ref{prop}, this follows from the result of
Dillencourt~\cite{Di90} that every Delaunay triangulation of $n$
points ($n$ even) has a perfect matching.
Further, we can shrink the $n/2$ disks until each disk is empty
and has the two points on its boundary. 

On the other hand, when the disks are required to be pairwise-disjoint,
a perfect matching need not exist.  \'Abrego~\etal~\cite{AA+05,AA+09} 
used the same circular construction as in our proof of Theorem~\ref{thm:174}
(details in Section~\ref{sec:circular})
to prove that for even $n \ge 74$, that particular $n$-element point 
set does not admit any perfect matching with $n/2$ pairwise-disjoint
empty disks. On the positive side, they showed that any $n$ points
in the plane (in general position) admit a matching of $n/4 - O(1)$ points
with pairwise-disjoint empty disks.
One can verify that for any $n$ points in the plane
(not necessarily in general position),
their proof implies that there exists a set of $n/8 - O(1)$ disjoint empty
disks incident to at least $n/4 - O(1)$ points such that
each disk is supported by at least two points.
As a consequence we have $D(n) \leq n/8 + 3n/4 + O(1) = 7n/8 + O(1)$.
This leads to the following upper bound complementing the lower bound
in Theorem~\ref{thm:236}.

\begin{proposition}
$D(n) \leq 7n/8 + O(1)$. 
\end{proposition}

To put our constructions in context,
we mention some constructions with disks
and squares which are build, similar to ours, by using the same principle
of ``exerting pressure''. For illustration, the reader is referred 
to the articles~\cite{Aj73,BDJ10a,BDJ08b,PT96}
for other constructions that we see related in this spirit.
\begin{figure}[htb]
\center\includegraphics[scale=0.95]{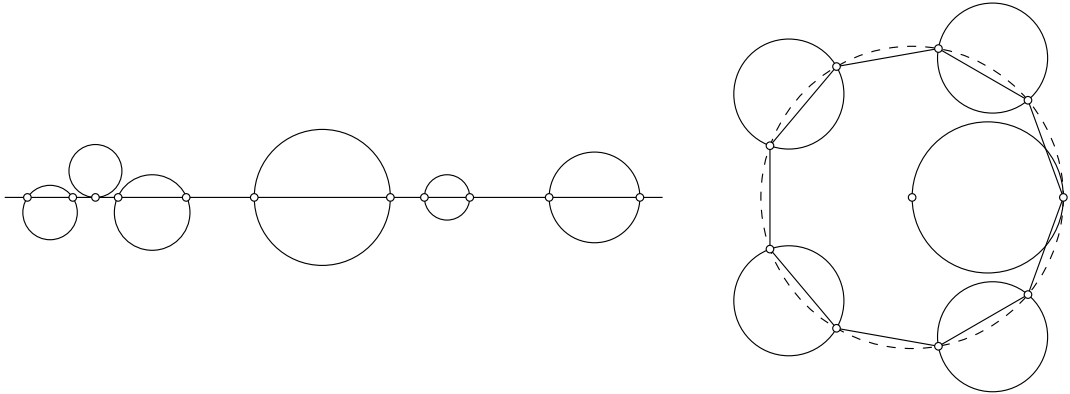}
\caption{Two constructions of Hosono and Urabe.
Left: a set of $n=11$ points on a line.
Right: a set of $n=10$ points consisting of the center and the vertices
of a regular $(n-1)$-gon.}
\label{fig:hosono_urabe}
\end{figure}

\paragraph{Preliminaries.}

In this paper we restrict ourselves to disjoint bubbles,
in particular to the estimation of $D(n)$. 
Note that since any two disks supported by a common point must
be interior-disjoint, each point supports at most two disks.
Since all disks we consider are empty, such a disk $Q$ is incident to
a point $p$ if and only if its boundary, $\partial Q$, is incident to
$p$. Before presenting our lower bound construction, we briefly review the
two constructions from~\cite{HU11} yielding the lower bound
$D(n) \geq n/2$. In fact, our constructions yielding Theorems~\ref{thm:174}
and~\ref{thm:236} were suggested by them. 

\begin{enumerate}
\item
Let $P$ be a collinear set of $n$ points; refer to
Figure~\ref{fig:hosono_urabe} (left). Any point must
support at least one empty disk, but no empty disk can be supported
by more than two points. This point set yields $B(P) \geq |P|/2$, 
thus $B(n) \geq B(P) \geq |P|/2 =n/2$, and by~\eqref{E1},
$D(n) \geq B(n) \geq n/2$.
\item
Assume that $n$ is even. Let $P$ consist of the vertices of a
regular $(n-1)$-gon together with its center (\ie, the center of the 
circumscribed circle of the regular $(n-1)$-gon); refer to
Figure~\ref{fig:hosono_urabe} (right).
The circle through any three vertices of the regular $(n-1)$-gon
must be its circumscribed circle,
which contains its center and hence is not empty.
It follows that any empty disk can be supported by
at most two (consecutive) vertices of the regular $(n-1)$-gon.
Thus for every even $n \ge 4$, $B(n) \ge B(P) \ge \lceil (n-1)/2 \rceil = n/2$,
and by~\eqref{E1}, $D(n) \geq B(n) \geq n/2$.
\end{enumerate}

\paragraph{Notation.}
An empty disk incident to only one point of $P$ is called a 
\emph{singleton disk}.
In turn, the unique point on a singleton disk is called a 
\emph{singleton point}. 

Given a disk $Q$, a horizontal line $\ell$, so that
$Q \cap \ell \neq \emptyset$, let $\xi(Q,\ell)$ denote the 
point of intersection between $\ell$ and the vertical diameter of $Q$.
For convenience, points on horizontal lines are frequently identified
with their $x$-coordinates when there is no danger of confusion.

\paragraph{Organization of the paper.}
In Section~\ref{sec:circular} we use a circular construction 
showing that the conjectured equality $D(n)=\lceil n/2 \rceil$ is already
false for $n$ as small as $174$. 
We then use this construction to obtain a lower bound of
$D(n) \geq n/2 + n/257- O(1)$ for every $n$.
In Section~\ref{sec:linear} we present a linear construction 
and its grid-generalization which achieves our record lower bound of 
$D(n) \geq n/2 + n/236- O(\sqrt{n})$ for every $n$.

\section{Circular constructions} \label{sec:circular}

For any $n \ge 4$, define an \emph{$n$-gadget} $G_n$ as a set of $n$ points
consisting of the center $p$ of a unit-radius circle $C$ and the vertices
$q_1,\ldots,q_{n-1}$ of a regular $(n-1)$-gon inscribed in the circle.
We call $p$ the \emph{center point} and $q_1,\ldots,q_{n-1}$ the 
\emph{boundary points}. We refer to Figure~\ref{fig:n12} for a
$12$-gadget. 
\begin{figure}[htb]
\centering\includegraphics[scale=0.95]{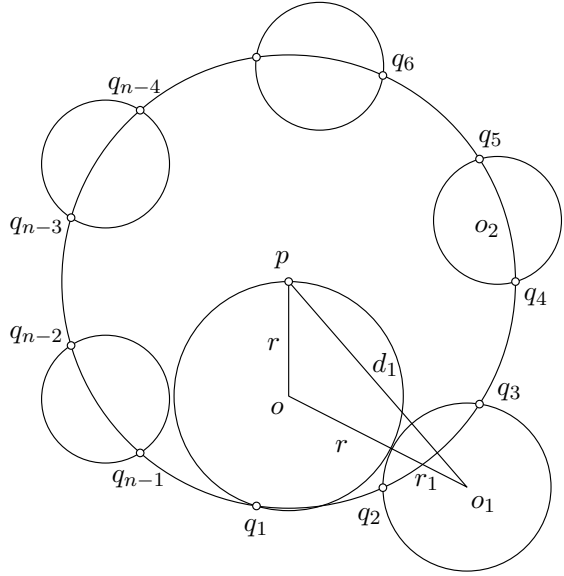}
\caption{An $n$-gadget with $n = 12$.}
\label{fig:n12}
\end{figure}

Hosono and Urabe noted that for any even $n \ge 4$,
the bubble number $B(G_n)$ of the $n$-gadget $G_n$ is at least $n/2$.
To see this, observe that the circle through any three points on the boundary
of $C$ must be $C$ itself,
which contains its center $p$ and hence is not empty.
It follows that any empty disk can be supported by
at most two (consecutive) points on the boundary of $C$.
Thus for every even $n \ge 4$,
the $n-1$ boundary points must support at least
$\lceil (n-1)/2 \rceil = n/2$ empty disks.

Hosono and Urabe then conjectured that for every $n$, $D(n) = \lceil n/2 \rceil$.
Our Theorem~\ref{thm:174} shows that this conjecture is false,
for even $n$ as small as $174$. Specifically, we prove that for every
even $n \ge 174$, $D(G_n) \ge n/2 + 1$.

\paragraph{Proof of Theorem~\ref{thm:174}.}
Consider any bubble set $\Q$ of disjoint empty disks for $G_n$.
We will show that $|\Q| \ge n/2 + 1$.
Recall that any empty disk can be supported by at most $2$ boundary points.
Thus the $n-1$ boundary points of $G_n$ support at least
$\lceil (n-1)/2 \rceil = n/2$ empty disks in $\Q$.
Let $Q \in \Q$ be an empty disk supported by the center point $p$ of $G_n$.
If $Q$ is not supported by any boundary point,
then the total number of disks in $\Q$ is at least $n/2 + 1$.
We next consider the two remaining cases:
\begin{enumerate}

\item
$Q$ is supported by the center point $p$ and only one boundary point, say $q_1$.
We claim that at least one of the four points $q_2,q_3,q_4,q_5$
and at least one of the four points $q_{n-1},q_{n-2},q_{n-3},q_{n-4}$
are singleton points.

\item
$Q$ is supported by the center point $p$ and two consecutive boundary points,
say $q_1$ and $q_2$.
We claim that at least one of the four points $q_3,q_4,q_5,q_6$
and at least one of the four points $q_{n-1},q_{n-2},q_{n-3},q_{n-4}$
are singleton points.

\end{enumerate}
If our claims in the two cases are valid, then,
excluding the at most two boundary points supporting $Q$
and two singleton points near $Q$,
the remaining at least $(n-1) - 2 - 2 = n-5$ boundary points would support
at least $\lceil (n-5)/2 \rceil = n/2 - 2$ empty disks.
Adding back the disk $Q$
and the two disks supported by the two singleton points,
the total number of empty disks would thus be at least
$(n/2 - 2) + 1 + 2 = n/2 + 1$, as desired.

We first consider the first part of the claim in case~1,
that at least one of the four points $q_2,q_3,q_4,q_5$ is a singleton point.
and prove it by contradiction.
Suppose to the contrary that none of $q_2,q_3,q_4,q_5$ is a singleton point.
Then there must exist a disk $Q_1 \in \Q$ supported by both $q_2$ and $q_3$,
and a disk $Q_2 \in \Q$ supported by both $q_4$ and $q_5$.
We will show that for every even $n \ge 174$,
the two disks $Q_1$ and $Q_2$ intersect,
which contradicts our assumption of disjoint empty disks in the bubble set.

Put $\theta = \pi/(n-1)$.
Let $r$ be the radius, and $o$ the center, of the disk $Q$.
For $i \in \{1,2\}$,
the radius $r_i$ of the disk $Q_i$ is uniquely determined by
the distance $d_i$ from the disk center $o_i$ to $p$,
according to the cosine rule:
\begin{equation}\label{eq:cosine1}
r_i^2 = 1^2 + d_i^2 - 2\cdot 1 \cdot d_i \cdot \cos\theta
= d_i^2 - 2 d_i \cos\theta + 1.
\end{equation}
Without loss of generality, we assume that
both $Q_1$ and $Q_2$ are tangent to $Q$.
Then the distance between $o$ and $o_i$ is exactly $r + r_i$.
Thus again by the cosine rule, we have
\begin{equation}\label{eq:cosine2}
(r + r_i)^2 = r^2 + d_i^2 - 2\cdot r \cdot d_i \cdot \cos \angle opo_i.
\end{equation}
Given $r$ and $\angle opo_i$ (these parameters will be fixed later 
in the argument, in equations~\eqref{eq:parameters} and~\eqref{eq:r}),
the distance $d_i$ and the radius $r_i$ for each $i \in \{1,2\}$ are uniquely
determined by the two equations \eqref{eq:cosine1} and \eqref{eq:cosine2}.
This can be seen by a morphing argument.
Observe that as $d_i$ changes continuously,
the portion of the disk $Q_i$ inside the circle $C$ also changes continuously:
the portion for a smaller $d_i$ properly contains the portion
for a larger $d_i$.
Subsequently, given $\angle o_1 p o_2$,
the distance $|o_1 o_2|$ between $o_1$ and $o_2$
is uniquely determined by applying the cosine rule once more:
\begin{equation}\label{eq:cosine3}
|o_1o_2|^2 = d_1^2 + d_2^2 - 2\cdot d_1 \cdot d_2 \cdot \cos \angle o_1 p o_2.
\end{equation}
To verify that $Q_1$ and $Q_2$ intersect,
it suffices to show $|o_1o_2| < r_1 + r_2$ for every even $n \ge 174$.

Before we proceed with the calculation,
observe that the second part of the claim in case~1,
that at least one of the four points $q_{n-1},q_{n-2},q_{n-3},q_{n-4}$
is a singleton point, can be proved by contradiction in a similar way,
where $Q_1$ is a disk supported by both $q_{n-1}$ and $q_{n-2}$,
and $Q_2$ is a disk supported by both $q_{n-3}$ and $q_{n-4}$.
In particular, equations \eqref{eq:cosine1},
\eqref{eq:cosine2}, and \eqref{eq:cosine3} still apply.
Moreover, the same argument also applies to case~2.
Indeed case~2 is the bottleneck case since the disk $Q$ in case~2
exerts less pressure on the four points near it on each side
than the disk $Q$ in case~1 does.
Since the two parts of the claim in case~2 are symmetric,
it suffices to do the calculation for the first part of the claim
in case~2 that at least one of $q_3,q_4,q_5,q_6$ is a singleton point.

We now do the calculation for this part,
where $Q$ is supported by both $q_1$ and $q_2$,
$Q_1$ is supported by both $q_3$ and $q_4$,
and
$Q_2$ is supported by both $q_5$ and $q_6$.
We clearly have
\begin{equation}\label{eq:parameters}
\angle opo_1 = 4\theta,\quad
\angle opo_2 = 8\theta,\quad
\angle o_1 p o_2 = 4\theta.
\end{equation}
Since $|pq_1| = |pq_2| = 1$, $|op| = |oq_1| = |oq_2| = r$,
and $\angle opq_1 = \angle opq_2 = \theta$,
we also have
\begin{equation}\label{eq:r}
r = 1/(2\cos\theta).
\end{equation}
Plug the parameters in
\eqref{eq:parameters}
and \eqref{eq:r}
into the three equations
\eqref{eq:cosine1},
\eqref{eq:cosine2},
and
\eqref{eq:cosine3}. 
One can verify that for $n = 174$,
the values in~\eqref{eq:values} satisfy the equations.
\begin{equation}\label{eq:values}
\begin{array}{cc}
\theta = 1.0404\ldots^\circ,\quad
r = 0.5000\ldots,\smallskip\\
d_1 = 1.0762\ldots,\quad
d_2 = 1.0113\ldots,\quad
r_1 = 0.0785\ldots,\quad
r_2 = 0.0215\ldots,\smallskip\\
|o_1 o_2| = 0.0997\ldots,\quad
r_1 + r_2 = 0.1000\ldots.
\end{array}
\end{equation}
In particular,
$|o_1 o_2| < r_1 + r_2$ for $n = 174$,
and hence for every even $n \ge 174$.
Thus for every even $n \ge 174$, the two disks $Q_1$ and $Q_2$ intersect,
giving the desired contradiction.
This completes the proof.
\qed

\begin{theorem}\label{thm:257} For every $n$, 
$D(n) \ge (1/2 + 1/257)n - O(1)$.
\end{theorem}

\begin{proof}
Let $m = \lfloor n / (174 + 340) \rfloor$, and let $\hat n = (174 + 340)m$.
Since $D(n)$ is a non-decreasing function in $n$,
it suffices to show that $D(\hat n) \ge (1/2 + 1/257)\hat n = \hat n / 2 + 2m$.
To this end, we will construct a set $P$ of $\hat n$ points,
and then show that $D(P) \ge \hat n / 2 + 2m$.
\begin{figure}[htb]
\centering\includegraphics[scale=0.9]{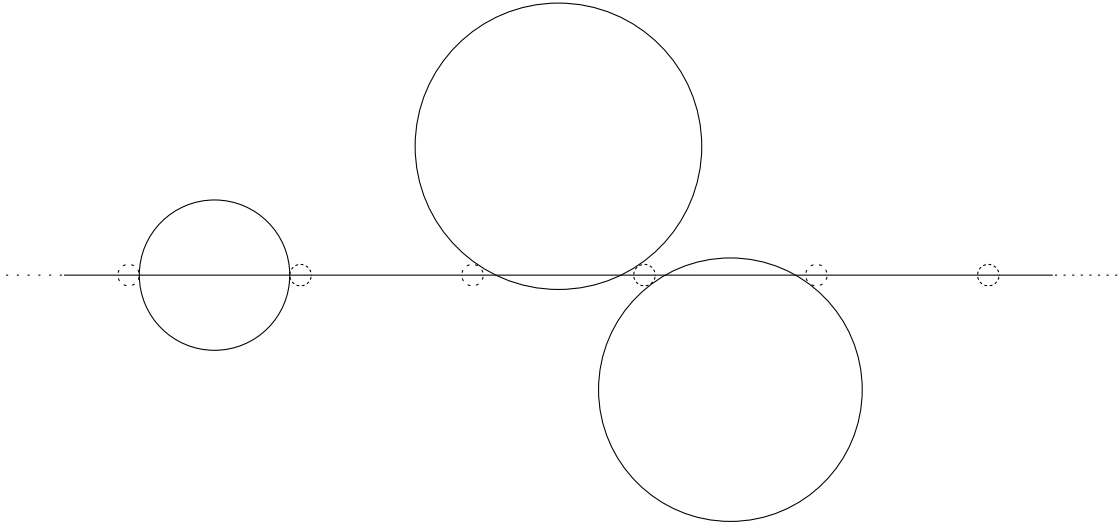}
\caption{The set $P$ consists of $m$ $174$-gadgets and $m$ $340$-gadgets
in an alternating pattern along a horizontal line.
Each bridge disk (three are shown here) is supported by two or more
points from two consecutive gadgets.} 
\label{fig:alternating}
\end{figure}

The set $P$ of $\hat n$ points are constructed as follows.
Refer to Figure~\ref{fig:alternating}.
Arrange $m$ $174$-gadgets and $m$ $340$-gadgets in an alternating
pattern along a horizontal line,
with a very large horizontal distance $d$ between the center points of
consecutive gadgets.
For each gadget, place the center point and one boundary point on the $x$-axis,
so that the remaining boundary points are symmetric about the $x$-axis
(this is possible because the number of points in each gadget is even).
Scale the $m$ $174$-gadgets very slightly such that all $2m$ gadgets are
sandwiched between two horizontal lines symmetric about the $x$-axis,
with each of the two lines containing exactly one boundary point
from each gadget
(again this is possible because the number of points in each gadget is even).

We proceed to show that $D(P) \ge \hat n / 2 + 2m$.
By construction, the scaling in particular, 
any disk supported by two points from two non-consecutive gadgets
cannot be empty.
Let $\Q$ be any bubble set of disjoint empty disks for $P$.
Then each disk in $\Q$ is either a \emph{local disk}
supported by one or more points from the same gadget,
or a \emph{bridge disk} supported by two or more points
from two consecutive gadgets, that is, a $174$-gadget and a $340$-gadget.
We charge each bridge disk to the $174$-gadget that supports it.
By Theorem~\ref{thm:174} (note that scaling does not affect the result),
each $174$-gadget supports at least $174/2 + 1$
disks in $\Q$ (both local and bridge). 
It remains to show that each $340$-gadget supports at least $340/2 + 1$
local disks in $\Q$.

Since the horizontal distance $d$ between the center points of consecutive
gadgets is very large, every bridge disk of radius $\Omega(d)$
appears as an approximately straight line near either gadget that supports it.
Consequently each $340$-gadget can support at most two disjoint bridge disks
that appear as two approximately parallel lines.
If a $340$-gadget does not support any bridge disk,
then by Theorem~\ref{thm:174} it supports at least $340/2 + 1$ local disks,
as desired.
We next consider a $340$-gadget supporting
either one bridge disk or two bridge disks.
\begin{figure}[htb]
\centering\includegraphics[scale=0.95]{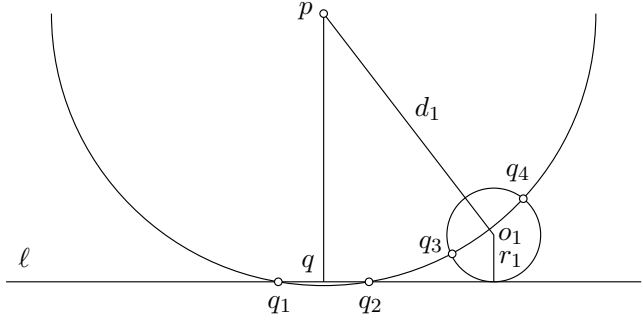}
\caption{A $340$-gadget supporting a bridge disk approximated by a line $\ell$
(the figure is not drawn in precise proportion;
the angles of the circular arcs between consecutive boundary points
are enlarged to show the details).}
\label{fig:ell}
\end{figure}

First consider a $340$-gadget supporting only one bridge disk.
We approximate the bridge disk by a straight line $\ell$ as justified earlier.
Refer to Figure~\ref{fig:ell}.
Without loss of generality,
we orient the construction such that $\ell$ is horizontal
and supports the gadgets from below.
We consider two cases analogous to the two cases in the proof of
Theorem~\ref{thm:174}: 
\begin{enumerate}

\item
$\ell$ is supported by only one boundary point, say $q_1$.
We claim that at least two of the four points $q_2,q_3,q_4,q_5$
and at least two of the four points $q_{339},q_{338},q_{337},q_{336}$
are singleton points.

\item
$\ell$ is supported by two consecutive boundary points, say $q_1$ and $q_2$.
We claim that at least two of the four points $q_3,q_4,q_5,q_6$
and at least two of the four points $q_{339},q_{338},q_{337},q_{336}$
are singleton points.

\end{enumerate}
As in the proof of Theorem~\ref{thm:174}, case~2 is the bottleneck case,
and by symmetry we only need to show that if $\ell$ is supported by
two consecutive boundary points $q_1$ and $q_2$,
then at least two of the four points $q_3,q_4,q_5,q_6$ are singleton points.
To show this by contradiction,
we use the same technique of two disks $Q_1$ and $Q_2$.
Suppose the contrary that at most one of the four points
$q_3,q_4,q_5,q_6$ is a singleton point.
Then we must have two disks $Q_1$ and $Q_2$ in one of the following three
subcases:
\begin{enumerate}

\item[i.]
$Q_1$ is supported by $q_3$ and $q_4$,
and $Q_2$ supported by $q_5$ and $q_6$;

\item[ii.]
$Q_1$ is supported by $q_3$ and $q_4$,
and $Q_2$ is supported by $q_6$ and $q_7$;

\item[iii.]
$Q_1$ is supported by $q_4$ and $q_5$,
and $Q_2$ is supported by $q_6$ and $q_7$.

\end{enumerate}

Put $k=340$, and $\theta = \pi/(k-1) = 0.5309\ldots^\circ$.
As in the proof of Theorem~\ref{thm:174},
let $r$ be the radius, and $o$ the center, of the disk $Q$.
For $i \in \{1,2\}$,
let $r_i$ be the radius, and $o_i$ the center, of the disk $Q_i$,
and let $d_i$ be the distance from $o_i$ to the center point $p$.
Then the two equations \eqref{eq:cosine1} and \eqref{eq:cosine3}
continue to hold.

In each subcase, we can assume that both $Q_1$ and $Q_2$ are tangent to $\ell$,
just as we assumed they are tangent to $Q$ in the proof of Theorem~\ref{thm:174}.
Let $q$ be the midpoint of $q_1$ and $q_2$.
Then $|pq| = \cos\theta$,
and we have the following equation analogous to~\eqref{eq:cosine2}:
\begin{equation}\label{eq:cosine2'}
d_i \cos\angle q p o_i + r_i = \cos\theta.
\end{equation}

We also have the following three sets of parameters,
analogous to the parameters in~\eqref{eq:parameters},
for the three subcases respectively:
\begin{align}
&\angle qpo_1 = 4\theta,\quad
\angle qpo_2 = \phantom{1}8\theta,\quad
\angle o_1 p o_2 = 4\theta,\label{eq:parameters1}\\
&\angle qpo_1 = 4\theta,\quad
\angle qpo_2 = 10\theta,\quad
\angle o_1 p o_2 = 6\theta,\label{eq:parameters2}\\
&\angle qpo_1 = 6\theta,\quad
\angle qpo_2 = 10\theta,\quad
\angle o_1 p o_2 = 4\theta.\label{eq:parameters3}
\end{align}
Plug the parameters in
\eqref{eq:parameters1}, \eqref{eq:parameters2}, \eqref{eq:parameters3},
respectively,
into the three equations
\eqref{eq:cosine1}, \eqref{eq:cosine2'}, \eqref{eq:cosine3}
(this is analogous to plugging \eqref{eq:parameters} into
\eqref{eq:cosine1}, \eqref{eq:cosine2}, \eqref{eq:cosine3}
in the proof of Theorem~\ref{thm:174}). 
One can now verify that
the following three sets of values in
\eqref{eq:values1}, \eqref{eq:values2}, \eqref{eq:values3},
respectively,
satisfy the equations.
\begin{equation}\label{eq:values1}
\begin{array}{cc}
d_1 = 0.9333\ldots,\quad
d_2 = 0.9854\ldots,\quad
r_1 = 0.0672\ldots,\quad
r_2 = 0.0172\ldots,\smallskip\\
|o_1 o_2| = 0.0631\ldots,\quad
r_1 + r_2 = 0.0845\ldots.
\end{array}
\end{equation}

\begin{equation}\label{eq:values2}
\begin{array}{cc}
d_1 = 0.9333\ldots,\quad
d_2 = 0.9919\ldots,\quad
r_1 = 0.0672\ldots,\quad
r_2 = 0.0122\ldots,\smallskip\\
|o_1 o_2| = 0.0794\ldots,\quad
r_1 + r_2 = 0.0795\ldots.
\end{array}
\end{equation}

\begin{equation}\label{eq:values3}
\begin{array}{cc}
d_1 = 0.9721\ldots,\quad
d_2 = 0.9919\ldots,\quad
r_1 = 0.0293\ldots,\quad
r_2 = 0.0122\ldots,\smallskip\\
|o_1 o_2| = 0.0414\ldots,\quad
r_1 + r_2 = 0.0415\ldots.
\end{array}
\end{equation}
Note that $|o_1 o_2| < r_1 + r_2$ for each of the three subcases.
Thus $Q_1$ and $Q_2$ intersect, giving the desired contradiction.

We have shown that at least $4$ of the $8$ boundary points near $\ell$
are singleton points.
Excluding the at most $2$ boundary points on $\ell$ and $4$ singleton points
near $\ell$, the remaining at least $(k-1) - 2 - 4 = k - 7$ boundary points
support at least $\lceil (k-7)/2 \rceil = k/2 - 3$ empty disks.
Adding back the $4$ empty disks supported by the $4$ singleton points,
the total number of local disks is at least $k/2 - 3 + 4 = k/2 + 1$.
That is, the $340$-gadget supports at least $340/2 + 1$ local disks.

Finally consider a $340$-gadget supporting two bridge disks.
Recall that the two bridge disks appear as two approximately parallel lines.
Thus the $8$ boundary points near one line are approximately opposite to
(and hence disjoint from) the $8$ boundary points near the other line.
And thus the same analysis for the two lines show that there are altogether
$8$ singleton points among the $16$ boundary points.
Excluding the at most $4$ boundary points on the two lines
and $8$ singleton points near the two lines,
the remaining at least $(k-1) - 4 - 8 = k - 13$ boundary points
support at least $\lceil (k-13)/2 \rceil = k/2 - 6$ empty disks.
Adding back the $8$ empty disks supported by the $8$ singleton points,
the total number of local disks is at least $k/2 - 6 + 8 = k/2 + 2$.
That is, the $340$-gadget supports at least $340/2 + 2$ local disks,
more than what we desired.

In summary, we have shown that
each $174$-gadget supports at least
$174/2 + 1$ (both local and bridge) disks in $\Q$,
and that
each $340$-gadget supports at least
$340/2 + 1$ local disks in $\Q$.
Summing up the number of disks over all $m$ $174$-gadgets and $m$ $340$-gadgets,
we have $D(P) \ge \hat n / 2 + 2m$.
This completes the proof of Theorem~\ref{thm:257}. 
\end{proof}

\section{Linear constructions} \label{sec:linear}

In Section~\ref{sec:prelim} we outline a preliminary (simpler)
construction that is easier to analyze. In Section~\ref{sec:refined}
we present a refinement which yields a better lower bound.

\subsection{Preliminary construction} \label{sec:prelim}

For every $n$, we construct a set $P$ of $n$ points, with
$D(P) \geq n/2 + n/966 - O(1)$.
We first describe a construction for every $n= C_1 j + C_2$,
$j=1,2\ldots$, for some suitable integers $C_1,C_2 \geq 1$.
We then extend the construction for every $n$. 

To achieve this bound we enforce that $\Omega(n)$ points of $P$
are singleton points.
Place the $n$ points on three parallel horizontal lines
$\ell_1,\ell_2,\ell_3$, as illustrated in Figure~\ref{f1}. 
The distance between $\ell_1$ and $\ell_3$ is $4$; the line in the
middle, $\ell_2$, is equidistant from $\ell_1$ and $\ell_3$. 
\begin{figure}[htb]
\centerline{\epsfxsize=6in \epsffile{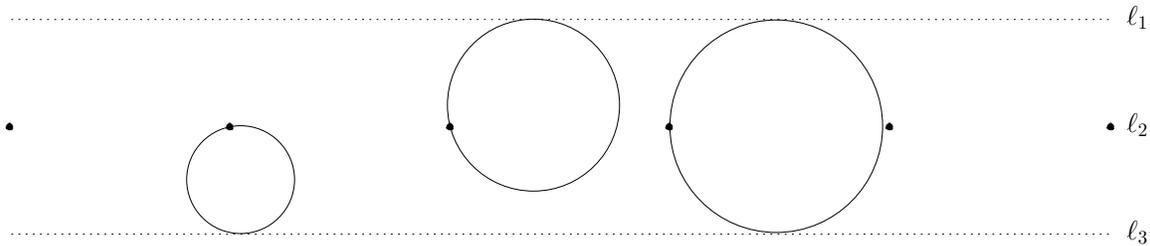}}
\caption{Construction and three bubbles: 
one supported by a point on $\ell_2$ and a point on $\ell_3$, 
one supported by a point on $\ell_2$ and two points on $\ell_1$, 
and one supported by a point on each of the three lines.} 
\label{f1}
\end{figure}

Let $0<\delta \leq 1/50$, so that $4/\delta$ is an integer; this
condition will be ensured by the choice of $\delta$.
Let $j \geq 1$ be a positive integer.
We first place the points on the three lines in a regular fashion.
Then shift right very slightly by the same distance $\eps$ those on
$\ell_3$, so that no four points of $P$ are co-circular on an empty circle.
Place $(4/\delta+1)j +1$ uniformly distributed points on $\ell_1$,
with any two consecutive points at distance $\delta$; proceed in the same
way for $\ell_3$.
Place $j+1$ uniformly distributed points on $\ell_2$,
labeled $p_1,\ldots,p_{j+1}$, 
with any two consecutive points at distance $4+\delta$. 
The leftmost points on the three lines are aligned vertically,
say at $x=0$. 
The small perturbation (shift) does not affect the points on 
$\ell_1$ and $\ell_2$.

We refer to $\ell_1$ and $\ell_3$ as the \emph{dense lines} 
and to the middle line $\ell_2$ as the \emph{sparse line}. 
The total number of points (on the three lines) is
\begin{equation} \label{E2}
n= \left( \frac{4}{\delta} + \frac{4}{\delta} + 3 \right) j +3
= \left( \frac{8}{\delta} + 3 \right) j +3. 
\end{equation} 

In an alternative view of this construction, the initial points 
(before perturbation) are placed on the
boundary and on the horizontal line through the center of the rectangle
$ R=[0,(4+\delta)j] \times [0,4]$. We refer to the $j-1$ points of $P$ in the
interior of $R$, \ie, on $\ell_2$, as {\em interior points}. 
As yet another view, the initial points are placed on the
boundaries of the $j$ axis-aligned rectangles 
$[(4+\delta)i, (4+\delta)(i+1)] \times [0,4]$, $i=0,\ldots,j-1$.

Let $\Q$ be a set of pairwise-disjoint bubbles (disks) supported by $P$.
Let $\Q'$ be the subset of disks that are either incident to some
interior point of $P$ or are incident to exactly three points of~$P$
but not to $p_1$ or $p_{j+1}$. 
Observe that the total number of disks in $\Q$ incident to $p_1$ or
$p_{j+1}$ is $O(1)$. 

First, observe that the radius of any empty disk incident to an interior
point and some other point is at least $1$. 
Second, observe also that if an empty disk whose center is in $R$ intersects
both $\ell_1$ and $\ell_3$, then its radius is at least $2$.
We refer to such disks (of the above two types) as \emph{large}, and
to any other disk as \emph{small}. 
Third, observe that the radius of any empty disk supported by an interior 
point is at most 
$$ \sqrt{16+\delta^2} \, / \, 2 < 2+\delta/2. $$
Similarly, the radius of any empty disk supported by 
a point in $\ell_1$ and a point in $\ell_3$ is also 
bounded from above by the same quantity. 
It follows that no empty disk can be incident to two interior points
(on $\ell_2$). 

The next two lemmas show that if an interior point $p \in P$ on the
sparse line supports a disk that is also supported by another point of $P$, 
then sufficiently many points on one of the two dense lines 
must be singleton points. Moreover these ``forced'' points can
be uniquely ``assigned'' to the corresponding disk incident to $p$. 
This suggests that in a minimum-size bubble set supported by $P$, 
each interior point must be a singleton point. A calculation making
this intuition precise is at the end of this subsection.

\begin{lemma}\label{L1}
Let $p_1,p_2,p_3,p_4$ be four consecutive points on the dense line $\ell_3$. 
Let $Q_{1,2}$ and $Q_{3,4}$ be two disks, the first supported by $p_1$
and $p_2$, and the second supported by $p_3$ and $p_4$. Denote by
$\ell'_3$ the horizontal line at distance $(1-\sqrt3/2)\delta$ above $\ell_3$. 
If both disks lie strictly below $\ell'_3$, then they overlap in their
interior (hence they are not disjoint). 
\end{lemma}
\begin{proof}
Refer to Figure~\ref{f2}.
Let $\omega_{1,2}$ be the disk supported by $p_1$ and $p_2$ and
tangent to $\ell'_3$. Similarly, 
let $\omega_{3,4}$ be the disk supported by $p_3$ and $p_4$ and
tangent to $\ell'_3$. 
Denote by $\pi_3$ the closed half-plane below $\ell_3$. 
We clearly have that 
$\omega_{1,2} \cap \pi_3 \subset Q_{1,2} \cap \pi_3$
and $\omega_{3,4} \cap \pi_3 \subset Q_{3,4} \cap \pi_3$.
Therefore, it suffices to show that $\omega_{1,2}$ and $\omega_{3,4}$ 
are tangent to each other, since then, it follows that 
$Q_{1,2}$ and $Q_{3,4}$ are not pairwise-disjoint. 
\begin{figure}[htb]
\centerline{\epsfxsize=2.3in \epsffile{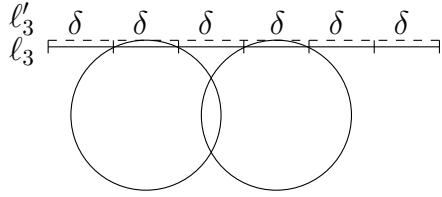}}
\caption{Proof of Lemma~\ref{L1}.} 
\label{f2}
\end{figure}

The two disks $\omega_{1,2}$ and $\omega_{3,4}$ are congruent,
and since they are both tangent to $\ell'_3$, their centers lie on 
the same horizontal line. It suffices to show that their radius $r$
is at least $\delta$.
It is well-known that the radius of a circle circumscribed to a
triangle with side lengths $a$,$b$,$c$ is equal to $abc/(4S)$, where $S$ stands
for the triangle area. It follows that
$$ r=\frac{\delta \left[ \left( \frac{\delta}{2} \right)^2 
+ \left((1-\frac{\sqrt{3}}{2})\delta \right)^2 \right]}
{(2-\sqrt3)\delta^2} = \delta. $$
Hence the two disks $\omega_{1,2}$ and $\omega_{3,4}$ are tangent to
each other, as desired, and this completes the proof.
\end{proof}

\begin{lemma}\label{L2}
Let $\ell \in \{\ell_1,\ell_3\}$. 
Let $Q \in \Q'$ be a large empty disk centered at $o$ that intersects $\ell$.
Then there exist two singleton points of $P$ on $\ell$ at horizontal
distance at most $7\delta/2$ from $o$, and that are not incident to $Q$. 
As such, these points can be uniquely assigned to $Q$ (and to no other
disk in $\Q$).  
\end{lemma}
\begin{proof}
Assume for simplicity that $\ell=\ell_3$, and put $u=\xi(Q,\ell_3)$.
Then $u$ is contained in the interval between two consecutive
points, say $q$ and $q'$.  Refer to Figure~\ref{f3} (left).
For simplicity suppose that $u$ is contained in the left 
half-interval of $qq'$ (the other case is symmetric). 
\begin{figure}[htb]
\centerline{\epsfxsize=6.2in \epsffile{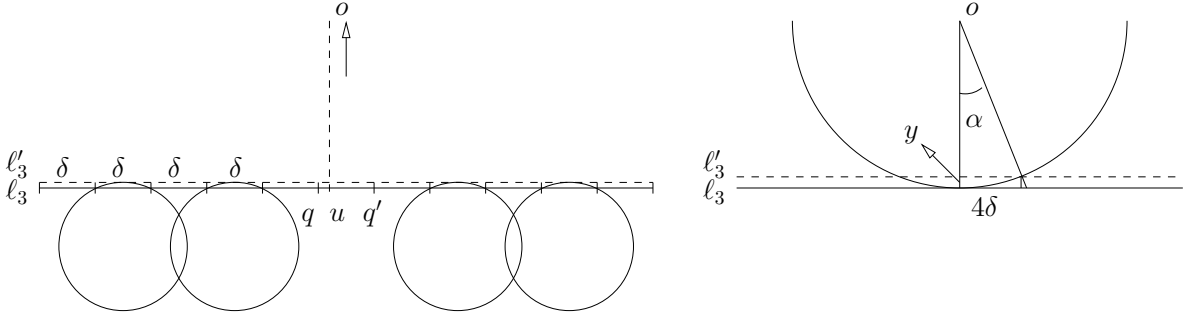}}
\caption{Proof of Lemma~\ref{L2}.} 
\label{f3}
\end{figure}

If $u$ is the midpoint of $qq'$, since $Q$ is empty, $Q$ is either
incident to both $q$ and $q'$ or to neither of them. 
If $u$ is not the midpoint of $qq'$, then $Q$ can be only incident to
$q$, but not to $q'$, if at all.

Let $q_i$, $i=1,\ldots,4$, be the
first four points of $P \cap \ell_3$ right of $u$ that are not
incident to $Q$. Observe that the farthest is at distance at most
$9\delta/2$ from $u$, hence the $x$-coordinate of the center of the
farthest disk incident to some pair $q_i,q_{i+1}$, $i=1,2,3$, 
is at most $4\delta$.
Let $r$ denote the radius of $Q$; obviously, $r \geq 1$. 
Let $y$ be the vertical distance between the midpoint of
$q_3 q_4$ and $Q$. 
It suffices to show that $y \leq 0.1337 \delta$, since then,
$y \leq 0.1337 \delta < (1 - \sqrt3/2)\delta$. 
Recall that $\delta \leq 1/50$, so we have 
\begin{align*}
y &\leq r -\sqrt{r^2 -16\delta^2} =
\frac{16\delta^2}{r+ \sqrt{r^2 -16 \delta^2}} 
\leq \frac{16\delta^2}{1+ \sqrt{1-16 \delta^2}} \\
&\leq \frac{16\delta^2}{1.997} \leq 8.02 \delta^2.
\end{align*}

Set now $\delta=1/60$, which ensures that
$ y \leq 8.02 \delta^2 \leq 0.1337 \delta$. 
Moreover, $4/\delta=240$ is an integer multiple of $\delta$, as desired.

It remains to show the existence of two singleton points 
in the vicinity of $u$. We show that there is at least one 
on each side of $u$. According to our previous calculation, 
any disk incident to some pair $q_i,q_{i+1}$, $i=1,2,3$, 
must lie below the line $\ell'_3$.
By Lemma~\ref{L1}, it follows that at least one of the points 
$q_i$, $i=1,\ldots,4$, is a singleton point. 
Note that since only consecutive points can be ``matched''
on an empty small circle, this implies that at least one of the points 
$q_i$, $i=1,\ldots,3$, is a singleton point. 
The farthest of these three points is at distance at most $7\delta/2$ from $u$.
Similarly, there is a singleton point left of $u$ 
at distance at most $7\delta/2$ from $u$.
Observe that the choice of these points may depend on the number of
points on $\ell_3$ (zero, one or two) incident to~$Q$. 
\end{proof}

Recall that by construction, each disk in $\Q$ is incident to at most
three points in $P$. Moreover, no disk in $\Q$ is incident to two
interior points (on $\ell_2$). 

\begin{lemma}\label{L3}
Let $Q \in \Q'$ be an empty disk. Then there exists a subset $\Q_1
\subset \Q$ of disks such that $Q \in \Q_1$ and $\Q_1$ is supported by
exactly $m$ points of $P$, where $|\Q_1| \geq (m+1)/2$. 
\end{lemma}
\begin{proof}
Assume that $Q \in \Q'$ is incident to (exactly) $i$ points of $P$.
By the previous observation, we know that $1 \leq i \leq 3$.
We distinguish two cases: 

\smallskip
\emph{Case 1.} $Q$ is incident to at least one interior point $p$ (on
$\ell_2$). 

If $i=1$, then set $\Q_1 =\{Q\}$, and the claimed inequality holds 
with $m=1$. 

If $i=2$, then $Q$ is also incident to a point on $q \in \ell_3$
(the case  $q \in \ell_1$ is symmetric). Then by Lemma~\ref{L2}
there exist two points $q_1,q_2$ in the neighborhood of $q$
(left and right of $q$), each incident to exactly one disk in $\Q$,
say $Q_1$ and $Q_2$. 
Then set $\Q_1 =\{Q,Q_1,Q_2\}$, and the claimed inequality holds with
$m=4$: $|\Q_1|=3 \geq (4+1)/2$. 

If $i=3$, assume first that $Q$ is also incident to two points $q,q' \in \ell_3$
(the case with two points on $\ell_1$ is symmetric). Then by Lemma~\ref{L2}
there exist two points (left of $q$ and right of $q'$),  
each incident to exactly one disk in $\Q$, say $Q_1$ and $Q_2$, and
such that $Q_1$ and $Q_2$ are each incident to exactly one point in $P$.
Then set $\Q_1 =\{Q,Q_1,Q_2\}$, and the claimed inequality holds with
$m=5$: $|\Q_1|=3 \geq (5+1)/2$. 
Second, assume that $Q$ is also incident to a point $q \in \ell_1$
and a point $q' \in \ell_3$. Again by Lemma~\ref{L2},
there exist two points on $\ell_1$ (left and right of $q$) 
and two points on $\ell_3$ (left and right of $q'$), 
each incident to exactly one disk in $\Q$. Let $Q_1,Q_2,Q_3,Q_4$ be
these four disks. Then set $\Q_1 =\{Q,Q_1,Q_2,Q_3,Q_4\}$, and the claimed
inequality holds with $m=7$: $|\Q_1|=5 \geq (7+1)/2$. 

\smallskip
\emph{Case 2.} $Q$ is not incident to any interior point (on $\ell_2$). 
Then, by the definition of $\Q'$, $Q$ is not incident to any point on $\ell_2$. 
It follows that $Q$ is incident to three points on $\ell_1 \cup \ell_3$. 
By Lemma~\ref{L2}, there exist four other points, two on
$\ell_1$ and two on $\ell_3$, each incident to exactly one disk in
$\Q$. Let $Q_1,Q_2,Q_3,Q_4$ be these four disks. Then set 
$\Q_1 =\{Q,Q_1,Q_2,Q_3,Q_4\}$, and the claimed 
inequality holds with $m=8$: $|\Q_1|=5 \geq (8+1)/2$. 

This completes our case analysis and the proof of the lemma.
\end{proof}

We now analyze the construction.
We have set $\delta=1/60$, thus $4/\delta=240$. 
Let $\Q_2 \subset \Q$ be the subset of disks constructed
from $\Q'$ by taking the union of all disks in $\Q_1$ 
from all cases in the analysis in the proof of Lemma~\ref{L3}.
Clearly $\Q' \subset \cup \Q_1 = \Q_2 \subset \Q$. 
Put $d=|\Q_2|$. Let $T \subset P$ be the set of points incident to disks
in $\Q_2$; write $t=|T|$. Observe that each interior point yields 
one inequality of the form $|\Q_1| \geq (m+1)/2$,
so there are at least $j-1$ such inequalities. 
The corresponding sets $\Q_1$ are pairwise disjoint by construction, thus
by adding up all these inequalities yields
$$ |\Q_2| \geq \frac{t+j-1}{2}. $$
Apart from $O(1)$ exceptions, each disk in $\Q \setminus \Q_2$ is
incident to at most two points in $P$. Consequently,
\begin{equation} \label{E18}
|\Q| = |\Q_2| + |\Q \setminus \Q_2| \geq \frac{t+j-1}{2} +
\frac{n-t}{2} -O(1) = \frac{n+j}{2} -O(1). 
\end{equation}

By substituting the value of $j$ resulting
from~\eqref{E2}, it follows that the number of disjoint empty disks in
any disjoint bubble set for $P$ is at least 
\begin{equation} \label{E19}
\frac{n}{2} + \frac{j}{2} -O(1) =
\frac{n}{2} + \frac{n}{16/\delta+6} -O(1) 
= \frac{n}{2} + \frac{n}{966} -O(1).
\end{equation}

Write $C_1=8/\delta+3$. 
To extend the construction for every $n$, we add a small cluster of 
at most $C_1-1$ collinear points on $\ell_3$, far away from the ``main''
construction above. Observe that there can be at most two disks
incident to both the main part and the additional cluster. 
Recall that each disk is incident to at most three points (this
condition can be maintained), and conclude that the previous analysis
yields the same bound. 
\qed

\subsection{A refined approach} \label{sec:refined}

We modify the previous construction and refine its analysis. 
These changes are motivated as follows.  

\begin{enumerate}

\item The lower bound can be raised by ``replicating'' the
construction vertically: use $k+1$ dense lines separated by 
$k$ sparse lines in between. The analysis however becomes more
involved. Our preliminary construction has $k=1$. 

\item In our previous analysis we only require that small disks
  incident to two consecutive points on $\ell_3$ and near a large disk
  intersecting $\ell_3$ lie below $\ell'_3$. Observe that such a
  triplet of disjoint disks can be enlarged until the three disks
  become pairwise tangent. Thus $\delta$ can be increased and this
  brings another improvement in the lower bound on $D(n)$. 

\end{enumerate}

\begin{lemma}\label{L4}
Let $x \leq 1/10$. 
Let $C$ be a disk of unit radius centered at $(0,1)$. 
Let $C_1$ be a disk incident to the points $(3x/2,0)$ and $(5x/2,0)$.
Let $C_2$ be a disk incident to the points $(7x/2,0)$ and $(9x/2,0)$.
If the three disks are pairwise tangent then 
$x=2 \sqrt{\lambda} =0.03486\ldots$, where $ \lambda = \frac{4}{945} (7-4\sqrt{3}) $
is the smaller solution of the quadratic equation
$ 893025 \lambda^2-52920 \lambda+16 = 0 $.
\end{lemma}
\begin{proof}
Let $C_1$ be centered at $(2x,-y_1)$, and $C_2$ be centered at
$(4x,-y_2)$, where $0<y_2<y_1$. 
Denote by $r_i$ the radius of $C_i$, $i=1,2$.
We have five unknowns, $x,r_1,r_2,y_1,y_2$, and five connecting equations: 
two expressing incidences and three expressing tangencies. Put $x=2z$.
\begin{align*} 
z^2+ y_1^2 &= r_1^2 \\
z^2+ y_2^2 &= r_2^2 \\
16 z^2+ (1+y_1)^2 &= (1+r_1)^2 \\
64 z^2+ (1+y_2)^2 &= (1+r_2)^2 \\
16 z^2+ (y_1-y_2)^2 &= (r_1+r_2)^2 
\end{align*}

After eliminating $y_1$ and $y_2$ the system becomes
\begin{align*} 
16 z^2+ \left(1+\sqrt{r_1^2- z^2}\right)^2 &= (1+r_1)^2 \\
64 z^2+ \left(1+\sqrt{r_2^2- z^2}\right)^2 &= (1+r_2)^2 \\
16 z^2 + \left( \sqrt{r_1^2- z^2} - \sqrt{r_2^2- z^2} \right)^2&= (r_1+r_2)^2 
\end{align*}

Recall that $z \leq r_1,r_2$, and make the substitutions
$$ r_1= \frac{z}{\cos \alpha}, \ \ r_2= \frac{z}{\cos \beta},
\textup{\ where \ } \alpha,\beta \in (0,\pi/2). $$
The above system can be rewritten as
\begin{align*} 
16 z^2+ (1+ z \tan \alpha)^2 &= \left(1+ \frac{z}{\cos \alpha}\right)^2. \\
64 z^2+ (1+ z \tan \beta)^2 &= \left(1+ \frac{z}{\cos \beta}\right)^2. \\
16 z^2+ (z  \tan \alpha - z \tan \beta)^2 &= 
z^2 \left( \frac{1}{\cos \alpha} + \frac{1}{\cos \beta} \right)^2. 
\end{align*}

Express $\cos \alpha$ and $\tan \alpha$ as functions of $\tan (\alpha/2)$:
\begin{equation*} 
\cos \alpha= \frac{1-s^2}{1+s^2}, \ \textup{and} \ 
\tan \alpha= \frac{2s}{1-s^2}, \ \textup{where} \ 
s=\tan \frac{\alpha}{2}. 
\end{equation*} 

Similarly, express $\cos \beta$ and $\tan \beta$ as functions of $\tan (\beta/2)$:
\begin{equation*} 
\cos \beta= \frac{1-t^2}{1+t^2}, \ \textup{and} \ 
\tan \beta= \frac{2t}{1-t^2}, \ \textup{where} \ 
t=\tan \frac{\beta}{2}. 
\end{equation*} 

Now the first equation of our system can be rewritten as
\begin{equation} \label{E4}
16 z^2+ \left(1+ z \frac{2s}{1-s^2} \right)^2 =
\left(1 + z \frac{1+s^2}{1-s^2} \right)^2.
\end{equation} 
By rearranging terms, \eqref{E4} is equivalent to
\begin{equation} \label{E5}
2(1-s^2)(1-s)^2=15z (1-s^2)^2.
\end{equation} 
The solution $s=1$ is infeasible, thus \eqref{E4} yields
\begin{equation} \label{E6}
s=\frac{2-15z}{2+15z}. 
\end{equation} 

Similarly, the second equation of our system can be rewritten as
\begin{equation} \label{E7}
64 z^2+ \left(1+ z \frac{2t}{1-t^2} \right)^2 =
\left(1 + z \frac{1+t^2}{1-t^2} \right)^2.
\end{equation} 
By rearranging terms, \eqref{E7} is equivalent to
\begin{equation} \label{E8}
2(1-t^2)(1-t)^2=63z (1-t^2)^2.
\end{equation} 
The solution $t=1$ is infeasible, thus \eqref{E8} yields
\begin{equation} \label{E9}
t=\frac{2-63z}{2+63z}. 
\end{equation} 

Recall that $s,t >0$, thus~\eqref{E6} and~\eqref{E9} require 
$z < 2/15$ and $z < 2/63$; overall, $z < 2/63$.
We can now express all trigonometric functions occurring in the third
equation of our system as functions of $z$. 
\begin{align*} 
\tan \alpha &= \frac{2s}{1-s^2}= \frac{4-225z^2}{60z}, \\
\tan \beta &= \frac{2t}{1-t^2}= \frac{4-3969z^2}{252z}, \\
\cos \alpha &= \frac{1-s^2}{1+s^2}= \frac{60z}{4+225z^2}, \\
\cos \beta &= \frac{1-t^2}{1+t^2}= \frac{252z}{4+3969 z^2}.
\end{align*}

Finally we substitute this in the third equation of the system 
$$ 16 + (\tan \alpha - \tan \beta)^2 = 
\left( \frac{1}{\cos \alpha} + \frac{1}{\cos \beta} \right)^2. $$
and obtain an equation in $z$ only:
\begin{equation} \label{E10}
16+ \left(\frac{4-225z^2}{60z} - \frac{4-3969z^2}{252z} \right)^2 =
\left(\frac{4+225z^2}{60z} + \frac{4+3969z^2}{252z} \right)^2,
\end{equation} 
or equivalently,
\begin{equation} \label{E11}
16 \cdot 1260^2 z^2 + (64+15120 z^2)^2 = (104 + 24570 z^2)^2.
\end{equation} 

Make the substitution $\lambda=z^2$ and (after simplifying by $42$) the
quartic equation in $z$~\eqref{E11} becomes a quadratic equation in $\lambda$:
\begin{equation} \label{E12}
893025 \lambda^2 -52920 \lambda +16=0.
\end{equation} 
Its solutions are 
$$ \lambda_{1,2} =\frac{4}{945} \, (7 \pm 4 \sqrt{3}). $$
Correspondingly, we have 
$$ z_{1,2}= \sqrt{ \frac{4}{945} \, (7 \pm 4 \sqrt{3})}. $$
The larger solution $z_1$ is infeasible, as it contradicts 
our assumption $z<2/63$. Finally 
$$ x=2z_2=2 \sqrt{ \frac{4}{945} \, (7 - 4 \sqrt{3})}
=0.03486\ldots. 
\qedhere $$
\end{proof}

\paragraph{Refined construction.}
Set $\delta = 1/29 < 2\sqrt{\lambda} =0.03486\ldots$, 
where $ \lambda = \frac{4}{945} (7-4\sqrt{3}) $.
The points on $\ell_1$ and $\ell_3$ are (initially) placed as before,
at multiples of $\delta$. 
The points on $\ell_2$ are (initially) placed as before, at multiples
of $4+\delta$. 
We now replicate the above construction vertically: use $k+1$ dense
lines separated by $k$ sparse lines in between. In total there are
$2k+1$ lines, labeled $\ell_1,\ldots, \ell_{2k+1}$. Write
$\L=\{\ell_1,\ldots, \ell_{2k+1}\}$. 
Finally shift right the points on every other odd-numbered line
$\ell_3,\ell_7,\ell_{11},\ldots$ by a small $\eps$ so that no four
points of $P$ are co-circular on an empty circle. 
Our preliminary construction has $k=1$. The total number of points on
the $2k+1$ lines is 
\begin{equation} \label{E17}
n= ((k+1)(4/\delta+1) + k) j + (2k+1)
= jk (4/\delta+2) + (4/\delta+1) j+ (2k+1).
\end{equation} 

For a fixed dense line $\ell \in \L$, we commonly denote the points
in $P$ on $\ell$ by $q_i$, where $i \in \NN$. 
Suppose a large disk intersects a dense line;
since almost the entire disk lies either above or below the line,
we say that the disk intersects $\ell$ \emph{from above} or 
\emph{from below}.

\begin{observation} \label{O1}
Let $\ell$ be one of the dense horizontal 
lines, and $I \subset \ell$ be an interval of length $|I|=8\delta$. 
Let $\{Q_1,\ldots,Q_h\}$ be large disks in $\Q$ intersecting $\ell$, so
that $\xi(Q_i,\ell) \in I$, for each $i=1,\ldots,h$. Then $h \leq 2$. 
Moreover, $Q_1$ and $Q_2$ intersect $\ell$ from different sides, one
from above and one from below. 
\end{observation}

\begin{observation} \label{O2}
Let $Q \in \Q$ be a large disk incident to at least one point in $P$
on a dense line $\ell$. Let $m_j$ denote the midpoint of the 
subdivision interval $[q_j, q_{j+1}]$, $j \in \NN$, on $\ell$. 

\smallskip
{\rm (i)} If $Q$ is incident to exactly one point on $\ell$, namely $q_i$, 
then $|q_i m_{i+1}| \leq 2\delta$ and $|m_{i-2} q_i | \leq 2\delta$.

\smallskip
{\rm (ii)} If $Q$ is incident to exactly two points on $\ell$, namely $q_i$ and
$q_{i+1}$, then $|m_i m_{i+2}| \leq 2\delta$ and $|m_{i-2} m_i | \leq 2\delta$.
\end{observation}

\begin{corollary}\label{C1}
Let $Q \in \Q$ be a disk incident to at least one point in $P$ on a
dense line $\ell$. 

\smallskip
{\rm (i)} If $Q$ is incident to exactly one point on $\ell$, namely $q_i$, 
then there exists a singleton point among $\{q_{i-3},q_{i-2},q_{i-1}\}$,
and one among $\{q_{i+1},q_{i+2},q_{i+3}\}$.

\smallskip
{\rm (ii)} If $Q$ is incident to exactly two points on $\ell$, namely $q_i$ and
$q_{i+1}$, then there exists a singleton point among $\{q_{i-3},q_{i-2},q_{i-1}\}$, 
and one among $\{q_{i+2},q_{i+3},q_{i+4}\}$.
\end{corollary}
\begin{proof}
We first prove statement (ii).
Let $Q_j$ denote a disk incident to the points $q_{j}$ and $q_{j+1}$. 
If there exist two disks $Q_{i+2}$ and $Q_{i+4}$ such that
the three disks $Q=Q_i$, $Q_{i+2}$, and $Q_{i+4}$ are pairwise disjoint,
shrink $Q$ until it becomes tangent to $\ell_3$ (if not already tangent).
The three disks remain pairwise disjoint. 
By Observation~\ref{O2}(ii), the horizontal distance between the
centers of any two consecutive disks is at most $2\delta$. 
However, by the choice of $\delta$ in Lemma~\ref{L4}, the three disks must be
pairwise tangent in the final configuration, a contradiction. 

Statement (i) is proved by a similar argument, using the two inequalities
in Observation~\ref{O2}(i).
\end{proof}

For two disks $Q_1,Q_2 \in \Q$ intersecting a common line $\ell$,
denote by $f(Q_1,Q_2)$ the number of points of $P$ in the interval
$[\xi(Q_1,\ell),\xi(Q_2,\ell)]$ that are not incident to  
$Q_1$ or $Q_2$. 

Define an undirected graph $G$ with vertex set the set of large disks
in $\Q$ incident to points on dense lines in $\L$. Two disks $Q_1,Q_2 \in
\Q$ are adjacent in $G$ if they are each incident to points of $P$
on a common dense line $\ell \in \L$, and $f(Q_1,Q_2) \in
\{0,1,2,3,4,5\}$. By Observation~\ref{O1}, if $Q_1,Q_2 \in \Q$ 
are adjacent in $G$ then the two large disks intersect the common line
one from below, and one from above. 

\begin{lemma}\label{L5}
Let $a \leq 6$ be a positive integer, and 
$\delta \leq 1/29 < 2\sqrt{\lambda} =0.03486\ldots$.
Define the function
$$ g(a,\delta) =\frac{a^2 \delta^2}{1+\sqrt{1-a^2 \delta^2}}. $$
Then $g()$ is an increasing function of $a$,
and $g(3,\delta) + g(6,\delta) \leq 0.9 \delta$. 
\end{lemma}
\begin{proof}
The first claim is easy to verify, so it remains to check the second
claim. We have
\begin{align*}
g(3,\delta) + g(6,\delta) &= \frac{9\delta^2}{1+\sqrt{1-9\delta^2}}
+ \frac{36\delta^2}{1+\sqrt{1-36\delta^2}} \\
&\leq \frac{9\delta^2}{1.99} + \frac{36\delta^2}{1.97} \leq
23\delta^2 \leq 0.9 \delta, 
\end{align*}
as required.
\end{proof}

\begin{lemma}\label{L6}
Let $\ell \in \L$ be a dense horizontal line, and 
$Q_1,Q_2 \in \Q$ be two large disks intersecting $\ell$ 
(one from above and one from below, with the center of $Q_1$ left of 
the center of $Q_2$), such that the pair $(Q_1,Q_2)$ is an edge in $G$. 
Then there exist $4$ singleton points on $\ell$ 
uniquely associated with this edge of $G$. 
\end{lemma}
\begin{proof}
Let $r_i \geq 1$ be the radius of $Q_i$, $i=1,2$. 
Let $q_i$ be the leftmost point on $\ell$ incident to 
$Q_1$ and $q_j$ be the rightmost point on $\ell$ incident to 
$Q_2$; clearly $i \leq j$. Write $f:=f(Q_1,Q_2)$ for simplicity. We
distinguish $6$ cases.    

\smallskip
\emph{Case 0.} $f=0$. It suffices to show that $q_{j+1}$ and $q_{j+2}$
are singleton points. 
By a symmetric argument $q_{i-1}$ and $q_{i-2}$ are then also
singleton points. 
We will show that a small disk incident to both
$q_{j+2}$ and $q_{j+3}$ cannot be in $\Q$. 
Obviously, the diameter of any such disk is at least $\delta$. 
The same argument will show that a small disk incident to both
$q_{j+1}$ and $q_{j+2}$ cannot be in $\Q$. 

Assume otherwise, for contradiction.
Observe that the horizontal distance between $\xi(Q_2,\ell)$ and 
the midpoint of $[q_{j+2},q_{j+3}]$ is at most $3\delta$.
Similarly, the horizontal distance between $\xi(Q_1,\ell)$ and 
the same midpoint is at most $5\delta$. 
We now verify that the ``sandwich'' $Q_1,Q_2$ forces the 
vertical diameter of a disk incident to both $q_{j+2}$ and $q_{j+3}$
to be smaller than $\delta$. Indeed, this vertical diameter is at most
$y_1+y_2$, where $y_i$, $i=1,2$, is the vertical distance between the
midpoint of $[q_{j+2},q_{j+3}]$ and the corresponding disk. We have
$$ y_1 \leq r_1 -\sqrt{r_1^2 -25 \delta^2} =
\frac{25\delta^2}{r_1+ \sqrt{r_1^2 -25 \delta^2}} \leq
\frac{25\delta^2}{1+ \sqrt{1-25 \delta^2}} =g(5,\delta), $$
and
$$ y_2 \leq r_2 -\sqrt{r_2^2 -9 \delta^2} =
\frac{9\delta^2}{r_2+ \sqrt{r_2^2 -9 \delta^2}} \leq
\frac{9\delta^2}{1+ \sqrt{1-9 \delta^2}} =g(3,\delta). $$

By Lemma~\ref{L5}, 
$$ y_1+y_2 = g(5,\delta) + g(3,\delta)
\leq g(6,\delta) + g(3,\delta) \leq 0.9\delta, $$
which contradicts the minimum-$\delta$ requirement on the diameter. 
It follows that $q_{j+1}$ and $q_{j+2}$  are singleton points.

\smallskip
\emph{Case 1.} $f=1$. As in Case 0, it suffices to show that 
a small disk incident to both $q_{j+2}$ and $q_{j+3}$ cannot be in $\Q$. 
Assume otherwise, for contradiction.
Observe that the horizontal distance between $\xi(Q_2,\ell)$ and 
the midpoint of $[q_{j+2},q_{j+3}]$ is at most $3\delta$.
Similarly, the horizontal distance between $\xi(Q_1,\ell)$ and 
the same midpoint is at most $6\delta$. 
As in Case 0, we check that 
$y_1 + y_2 \leq g(6,\delta) + g(3,\delta) \leq 0.9\delta$, 
which holds by Lemma~\ref{L5}.

\smallskip
\emph{Case 2.} $f=2$. By Corollary~\ref{C1}, $q_{i-1}$ and $q_{j+1}$ are 
singleton points. It suffices to show that the two points of $P$ on
$\ell$ in between $Q_1$ and $Q_2$ are singleton points. 
The relevant inequality is
$g(2,\delta) + g(2,\delta) \leq 0.9\delta$, 
which holds by Lemma~\ref{L5}.

\smallskip
\emph{Case 3.} $f=3$. By Corollary~\ref{C1}, $q_{i-1}$ and $q_{j+1}$ are 
singleton points. It suffices to show that the three points of $P$ on
$\ell$ in between $Q_1$ and $Q_2$ are singleton points. 
The relevant inequality is
$g(2,\delta) + g(3,\delta) \leq 0.9\delta$, 
which holds by Lemma~\ref{L5}.

\smallskip
\emph{Case 4.} $f=4$. 
It suffices to show that the four points of $P$ on
$\ell$ in between $Q_1$ and $Q_2$ are singleton points. 
The relevant inequalities are
$g(2,\delta) + g(4,\delta) \leq 0.9\delta$, and
$g(3,\delta) + g(3,\delta) \leq 0.9\delta$, which hold by Lemma~\ref{L5}.

\smallskip
\emph{Case 5.} $f=5$. 
It suffices to show that the five points of $P$ on
$\ell$ in between $Q_1$ and $Q_2$ are singleton points. 
The relevant inequalities are
$g(2,\delta) + g(5,\delta) \leq 0.9\delta$, and
$g(3,\delta) + g(4,\delta) \leq 0.9\delta$, which hold by
Lemma~\ref{L5}.

The end of our case analysis concludes the proof of the lemma.
\end{proof}

The connected components of $G$ are vertex-disjoint paths;
see~Figure~\ref{f4} for an example. 
\begin{figure}[htb]
\centerline{\epsfxsize=4.5in \epsffile{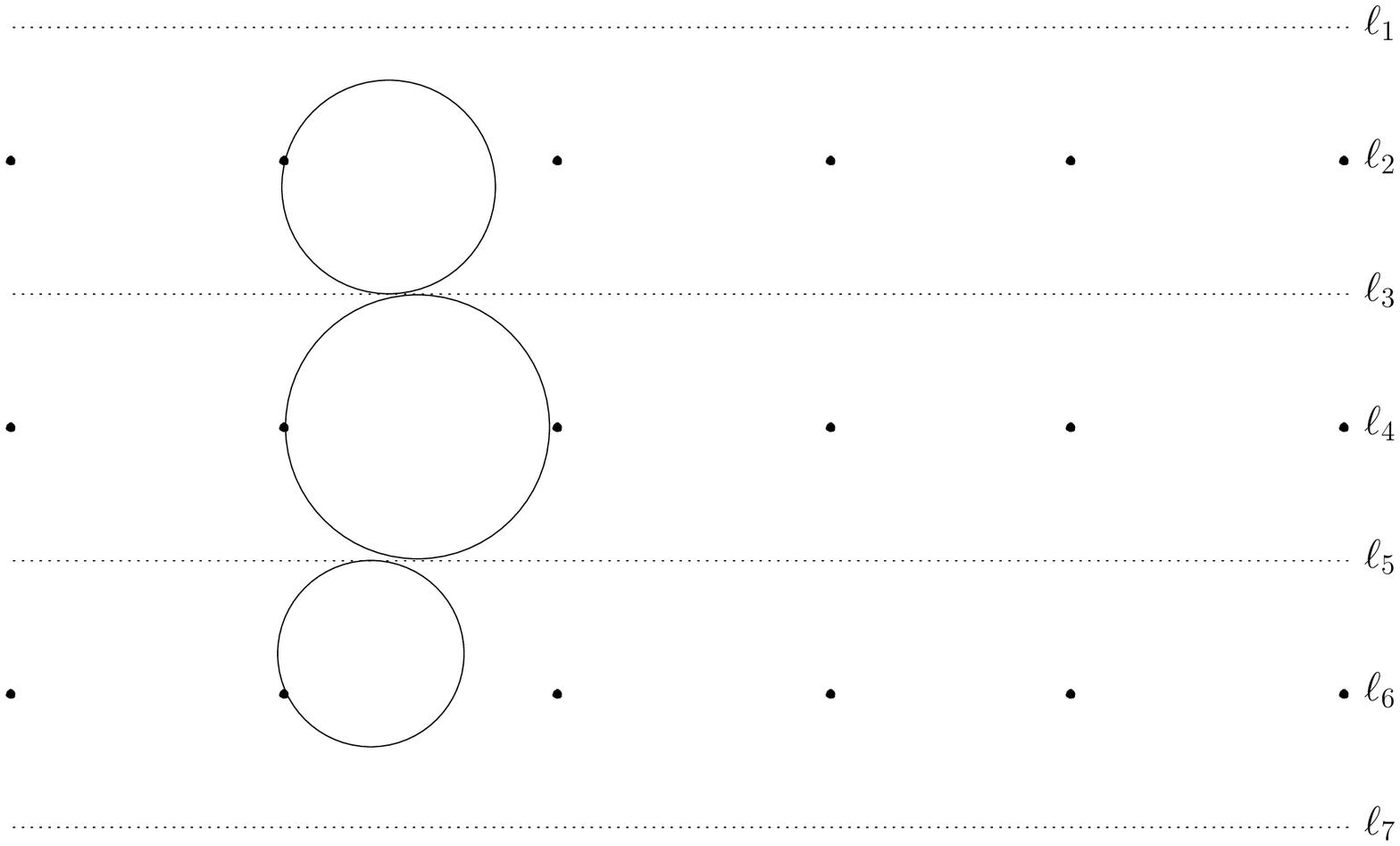}}
\caption{Construction and a $3$-vertex path in $G$.} 
\label{f4}
\end{figure}

\begin{lemma}\label{L7}
Let $C$ be a connected component of $G$ with $i$ vertices.
Then there exists a subset $\Q_1 \subset \Q$ of disks 
such that $C \subset \Q_1$ and $\Q_1$ is supported by exactly $m$ points
of $P$, where $|\Q_1| \geq (m+i)/2$. 
\end{lemma}
\begin{proof}
Recall that $C$ is a path of $i$ vertices in $G$, for some $i \geq 1$. 
If $i \geq 2$, each of the $i-1$ edges of this path is uniquely
associated with a dense line in $\L$. By Lemma~\ref{L6}, for each edge
in $C$ there exist $4$ singleton points uniquely associated with this edge.
Let $\Q_1$ consist of the $i$ large disks (in $C$) and the $4(i-1)$
singleton disks associated  with the $i-1$ edges of $C$. 
Thus $|\Q_1|=i+4(i-1)=5i-4$. Since each of the $i$ large disks 
in $C$ is incident to at most $3$ points in $P$, we have
$ m \leq 3i + 4(i-1)=7i-4$. It is now easy to verify that
$|\Q_1| = 5i-4 \geq 4i-2= ((7i-4)+i)/2 \geq (m+i)/2$, as required.
Note that equality is possible when $i=2$. 
\end{proof}

We can now finalize the proof of Theorem~\ref{thm:236} as follows.
We have set $\delta=1/29$, thus $4/\delta=116$. 
Consider a connected component of $G$, say involving 
$i$ large disks incident to at most $i$ sparse points. 
(Recall that no large disk can be incident to two sparse points.)
Let $\Q_2 \subset \Q$ be the subset of disks constructed
from $\Q'$ by taking the union of all disks in $\Q_1$ 
in the proof of Lemma~\ref{L7}.
Clearly $\Q' \subset \cup \Q_1 = \Q_2 \subset \Q$. 
Put $d=|\Q_2|$. Let $T \subset P$ be the set of points incident to disks
in $\Q_2$; write $t=|T|$. Observe that each connected component
yields one inequality of the form $|\Q_1| \geq (m+i)/2$,
so each interior point generates a ``surplus'' of $1/2$.
The corresponding sets $\Q_1$ are pairwise disjoint by construction, thus
by adding up all these inequalities yields
$$ |\Q_2| \geq \frac{t+(j-1)k}{2}. $$
Apart from $O(k)$ exceptions, each disk in $\Q \setminus \Q_2$ is
incident to at most two points in $P$. Consequently,
\begin{equation} \label{E20}
|\Q| = |\Q_2| + |\Q \setminus \Q_2| \geq 
\frac{t+(j-1)k}{2} + \frac{n-t}{2} -O(k) =
\frac{n+(j-1)k}{2} -O(k). 
\end{equation}

\paragraph{Analysis.}
We let $j=\Theta(\sqrt{n})$, and $k=\Theta(\sqrt{n})$, 
satisfying~\eqref{E17}. Analogous to~\eqref{E18} and~\eqref{E19}, 
by substituting the value of $(j-1)k$ resulting from~\eqref{E17}, 
it follows that the number of disjoint empty disks in a minimizing bubble
set is at least  
\begin{align*}
\frac{n}{2} + \frac{(j-1)k}{2} -O(k) =
\frac{n}{2} + \frac{n}{8/\delta+4} -O(j+k) =
\frac{n}{2} + \frac{n}{236} -O(\sqrt{n}).
\end{align*} 

This construction can be also extended for every $n$, by adding a
small cluster of at most $O(\sqrt{n})$ collinear points on one of the
lines. The above lower bound is preserved and this ends the proof of
Theorem~\ref{thm:236}. 
\qed

\section{Concluding remarks}

Via suitable slight perturbations, our lower bound constructions can be
realized with points in general position, for instance with no three
on a line and no four on a circle.  
However, the analysis needs to be adapted, and we leave the details to
the reader.

In our linear construction points are distributed uniformly
on the dense lines. One can reduce the number of points 
on these lines by making a non-uniform distribution,
with points farther away from each other when $x$ is close to odd
multiples of $2+\delta/2$. The reason is that the radius of a disk in $\Q$
incident to an interior point $p \in \ell_2$ and intersecting
$\ell_1$ or $\ell_3$ farther away horizontally from $p$ must be
significantly larger than $1$, a value which was used for simplicity
in our calculation. We however opted to leave out the non-uniform
distribution due to its complicated analysis.

While we believe it is possible to further raise the lower
bound $D(n) \geq n/2 + n/236 -O(\sqrt{n})$, \eg, replace 236 by 190 
by making this and other changes, making progress beyond this bound
probably requires new ideas. 
From the other direction, it is very likely that the 
upper bound $7n/8 +O(1)$ can be substantially reduced.

\paragraph{Acknowledgement.}
We thank an anonymous source for bringing the articles of
Dillencourt~\cite{Di90} and \'Abrego~\etal~\cite{AA+05,AA+09}
to our attention and for indicating that Proposition~\ref{prop} follows
from the result of Dillencourt~\cite{Di90}.

\end{document}